\documentclass[11pt]{amsart}
\usepackage{amssymb,amsmath}

\makeatletter
\theoremstyle{plain}
\newtheorem{thm}{Theorem}[section]
\newtheorem*{thm*}{Theorem}

\newtheorem{prop}[thm]{Proposition}
\newtheorem{cor}[thm]{Corollary}

\theoremstyle{definition}

\newtheorem{rem}[thm]{Remark}

\numberwithin{equation}{section}
\numberwithin{figure}{section}

\newcommand{\cF}{{\mathcal F}}
\newcommand{\cP}{{\mathcal P}}
\newcommand{\fC}{{\mathfrak C}}

\newcommand{\cM}{{\mathcal M}}
\newcommand{\cG}{{\mathcal G}}
\newcommand{\calF}{{\mathcal F}}
\newcommand{\C}{{\mathbb C}}

\newcommand{\K}{{\mathbb K}}
\newcommand{\R}{{\mathbb R}}

\newcommand{\rX}{{\rm X}}
\newcommand{\bH}{{\bf H}}

\newcommand{\bp}{{\bf p}}

\newcommand{\bx}{{\bf x}}
\newcommand{\by}{{\bf y}}

\newcommand{\bw}{{\bf w}}
\newcommand{\bX}{{\bf X}}
\newcommand{\cC}{{\mathcal C}}

\newcommand{\fp}{{\mathfrak p}}
\newcommand{\ft}{{\mathfrak t}}

\newcommand{\fx}{{\mathfrak x}}
\newcommand{\fy}{{\mathfrak y}}
\newcommand{\fc}{{\mathfrak c}}

\newcommand{\fF}{{\mathfrak F}}

\newcommand{\X}{{\mathbb X}}
\newcommand{\F}{{\mathbb F}}

\newcommand{\T}{{\mathbb T}}

\newcommand{\fM}{{\mathfrak M}}

\newcommand{\cQ}{{\mathcal Q}}
\newcommand{\cS}{{\mathcal S}}

\makeatother

\begin{document}

\title[Configuration space of four points in the torus]{The ${\rm PSL}(2,\R)^2$-configuration space of four points in the torus $S^1\times S^1$}

\author[I.D. Platis]{Ioannis D. Platis}
\email{idplatis@upatras.gr}
\address{Department of Mathematics\\ 
University of Patras\\
University Campus\\
GR 265 04 Rion\\ Achaia\\Greece}

\date{\today}

\begin{abstract}
The torus $\T=S^1\times S^1$ appears as the ideal boundary $\partial_\infty AdS^3$ of the three-dimensional anti-de Sitter space $AdS^3$, as well as the Fürstenberg boundary $\F(X)$ of the rank-2 symmetric space $X={\rm SO}_0(2,2)/{\rm SO}(2)\times{\rm SO}(2)$. We introduce cross-ratios on the torus in order to parametrise the ${\rm PSL}(2,\R)^2$ configuration space of quadruples of pairwise distinct points in $\T$ and define a natural M\"obius structure on $\T$ and therefore on $\F(X)$ and $\partial_\infty AdS^3$ as well. 
\end{abstract}

\maketitle

\section{Introduction}
Let $S$ be a topological space and denote by $\cC_4=\cC_4(S)$ the space of quadruples of pairwise distinct points of $S$. Let $G$ be a group of homeomorphisms of $S$ acting diagonally on $\cC_4$ from the left: for each $g\in G$ and $\fp=(p_1,p_2,p_3,p_4)\in\cC_4$,
$$
(g,\fp)\mapsto g(\fp)=(g(p_1),g(p_2),g(p_3),g(p_4)).
$$
The $G$-{\it configuration space of quadruples of pairwise distinct points in $S$} is the quotient space $\cF_4=\cF_4(S)$ of $\cC_4$ cut by the action of $G$. In this general setting, it is not obvious what kind of object $\cF_4$ is; however, there are tractable cases. If for instance $S$ is a smooth manifold of dimension $s$ and $G$ is a Lie subgroup of diffeomorphisms of $S$ of dimension $g$, then $\cF_4$ carries the structure of a smooth manifold of dimension $4s-g$ if the diagonal action is proper and free. 

The particular cases when $S$ are spheres bounding hyperbolic spaces and $G$ are the sets $\cM(S)$ of M\"obius transformations acting on $S$ are prototypical; in some of these cases there exist neat parametrisations of the configuration spaces by cross-ratios. The most illustrative (and simplest) example of all is that of the $\cM(S^1)$-configuration space $\cF_4(S^1)$ of quadruples of pairwise distinct points in the unit circle $S^1$. What follows is classical and well-known, see for instance \cite{Be}, but we include it here both for clarity as well as for setting up the notation we shall use throughout the paper. 

Let $S^1=\overline{\R}$, where $\overline{\R}=\R\cup\{\infty\}$, be the unit circle. Recall that if $\fx=(x_1,x_2,x_3,x_4)\in\cC_4(S^1)$, then its real cross-ratio is defined by
$$
\rX(\fx)=[x_1,x_2,x_3,x_4]=\frac{(x_4-x_2)(x_3-x_1)}{(x_4-x_1)(x_3-x_2)},
$$
where we agree that if one of the points is $\infty$, then $\infty:\infty=1$. The set of M\"obius transformations $\cM(S^1)$ of $S^1$ comprises maps $g:S^1\to S^1$ of the form
$$
g(x)=\frac{ax+b}{cx+d},\quad x\in\overline{\R},
$$
where the matrix
$$
A=\begin{pmatrix}
a&b\\
c&d\end{pmatrix}
$$
is in ${\rm PSL}(2,\R)={\rm SL}(2,\R)/\{\pm I\}$. The cross-ratio $\rX$ is invariant under the diagonal action of $\cM(S^1)$ in $\cC_4=\cC_4(S^1)$: If $g\in\cM(S^1)$, then
$$
\rX(g(\fx))=[g(x_1),g(x_2),g(x_3),g(x_4)]=[x_1,x_2,x_3,x_4]=\rX(\fx),
$$
for every $\fx\in\cC_4$. Also, $\rX$ takes values in $\R\setminus\{0,1\}$ and for each quadruple $(x_1,x_2,x_3,x_4)$ satisfies the standard symmetry properties:
\begin{align*}
({\rm S1})&\quad \rX(x_1,x_2,x_3,x_4)=\rX(x_2,x_1,x_4,x_3)=\rX(x_3,x_4,x_1,x_2)=\rX(x_4,x_3,x_2,x_1),\\
({\rm S2})&\quad \rX(x_1,x_2,x_3,x_4)\cdot \rX(x_1,x_2,x_4,x_3)=1,\\
({\rm S3})&\quad \rX(x_1,x_2,x_3,x_4)\cdot \rX(x_1,x_4,x_2,x_3)\cdot \rX(x_1,x_3,x_4,x_2)=-1.
\end{align*}
Hence all 24 real cross-ratios corresponding to a given quadruple $\fx$ are eventually functions of $\rX(\fx)=[x_1,x_2,x_3,x_4]$. We let $\cM(S^1)$ act diagonally on $\cC_4(S^1)$ and let $\calF_4=\calF_4(S^1)$ be the $\cM(S^1)$-configuration space of quadruples of pairwise distinct points in $S^1$. From invariance by cross-ratios we have that the map
$$
\cG:\calF_4(S^1)\ni[\fx]\mapsto \rX(\fx)\in\R\setminus\{0,1\},
$$
is well-defined. Also $\cG$ is surjective; to see this, recall that there is a triply-transitive action of $\cM(S^1)$ on $S^1$: if $(x_1,x_2,x_3)$ is a triple of pairwise distinct points in $S^1$, then there is a unique $f\in\cM(S^1)$ such that
$$
f(x_1)=0,\quad f(x_2)=\infty,\quad f(x_3)=1.
$$
Recall at this point that actually $f$ is given in terms of cross-ratios:
$$
[0,\infty,f(x),1]=[x_1,x_2,x,x_3].
$$  
Hence if $x\in\R\setminus\{0,1\}$, then $[\fx]\mapsto x$, where $\fx=(0,\infty,x,1)$. Finally, $\cG$ is injective: If $\fx$ and $\fx'$ are in $\cC$ and $\rX(\fx)=\rX(\fx')=x$, then there exists a $g\in\cM(S^1)$ such that $\fx=g(\fx')$. All the above discussion boils down to the well-known fact that the configuration space $\calF_4=\calF_4(S^1)$ of quadruples of pairwise distinct points in $S^1$ is isomorphic to $\R\setminus\{0,1\}$, and therefore it inherits the structure of a one-dimensional disconnected real manifold. Moreover, the following possibilities occur for the relative position of the points $x_i$ of $\fx$ on the circle:
\begin{enumerate}
\item $x_1,x_2$ separate $x_3,x_4$. This happens if and only if $\rX(\fx)<0$.
\item $x_1,x_3$ separate $x_2,x_4$. This happens if and only if $\rX(\fx)>1$.
\item $x_1,x_4$ separate $x_2,x_3$. This happens if and only if $0<\rX(\fx)<1$.
\end{enumerate}
Each of cases (1), (2) and (3) corresponds to the connected components of $\calF_4$.

In an analogous manner, see again \cite{Be}, the ${\rm PSL}(2,\C)$-configuration space $\cF_4(S^2)$ of quadruples of pairwise distinct points in the sphere $S^2$ is isomorphic to $\C\setminus\{0,1\}$ and therefore inherits the structure of a one-dimensional complex manifold. The case of the ${\rm PU}(2,1)$-configuration space $\cF_4(S^3)$ of quadruples of pairwise distinct points in $S^3$ is much harder but it is treated in the same spirit, see \cite{FP}: Using complex cross-ratios we find that besides a subset of lower dimension, $\cF_4(S^3)$ is isomorphic to $(\C\setminus\R)\times\C P^1$, a two-dimensional disconnected complex manifold. Finally, recent treatments for the cases of ${\rm PSp}(1,1)$ and ${\rm PSp}(2,1)$-configuration spaces of quadruples of pairwise distinct points in $S^3$ and $S^7$ may be found in \cite{GM} and \cite{C}, respectively. 

As we have mentioned above, spheres may be viewed as boundaries of symmetric spaces of non-compact type and of rank-1; that is, hyperbolic spaces $\bH_\K^n$, where $\K$ can be: a) $\R$ the set of real numbers, b) $\C$ the set of complex numbers, c) $\mathbb{H}$ the set of quaternions, and d) $\mathbb{O}$ the set of octonions (in the last case $n=2$). Two problems arise naturally here; first, to describe configuration spaces of four points on products of such spheres by parametrising them with cross-ratios defined on those products, and second, to describe configuration spaces of four points in boundaries of symmetric spaces of rank $>1$ again by parametrising them using cross-ratios defined on those boundaries. These two problems are sometimes intertwined and the crucial issue here is the definition of an appropriate cross-ratio; this is directly linked to the M\"obius geometry of the spaces we wish to study, as we explain below. In this paper we deal with both problems by describing in the manner above the configuration space of four points in the torus $\T=S^1\times S^1$; the torus is the Fürstenberg boundary of the symmetric space ${\rm SO}_0(2,2)/{\rm SO}(2)\times{\rm SO}(2)$ which is of rank-2, as well as the ideal boundary of anti-de Sitter space $AdS^3$, see Section \ref{sec:cons}. 

Returning to our original general setting, suppose that the $G$-configuration space $\cF_4(S)$ has a real manifold structure due to a proper and free action of $G$ on $\cC_4(S)$. By taking the product $S\times S$ and the space $\cC_4(S\times S)$ of quadruples of pairwise distinct points of $S\times S$, the group $G\times G$ acts diagonally as follows: for $g=(g_1,g_2)$ and $\fp=(p_1,p_2,p_3,p_4)\in\cC_4(S\times S)$, $p_i=(x_i,y_i)$, $i=1,2,3,4$,
$$
(g,\fp)\mapsto g(\fp)=\left((g_1(x_1),g_2(y_1)),(g_1(x_2),g_2(y_2)),(g_1(x_3),g_2(y_3)),(g_1(x_4),g_2(y_4))\right).
$$ 
Using elementary arguments, one deduces that the action of $G\times G$ on $\cC_4(S\times S)$ is proper. Towards the direction of the free action of $G\times G$ we observe that from the obvious injection
$$
\cC_4(S)\times\cC_4(S)\to\cC_4(S\times S)
$$
which assigns to each $(\fx,\fy)\in\cC_4(S)\times\cC_4(S)$, $\fx=(x_1,x_2,x_3,x_4)$, $\fy=(y_1,y_2,y_3,y_4)$, the quadruple $\fp=(p_1,p_2,p_3,p_4)$ where $p_i=(x_i,y_i)$, $i=1,2,3,4$, we obtain an injection
$$
\cF_4(S)\times\cF_4(S)\ni([\fx],[\fy])\mapsto [\fp]\in\cF_4(S^1\times S^1).
$$
The image of this map is the subset $\cF_4^\sharp(S\times S)$ of $\cF_4(S\times S)$ comprising quadruples $\fp$ such that $\fx$, $\fy$ are in $\cF_4(S)$. We may straightforwardly show that $\cF_4^\sharp$ comprises principal orbits, that is, orbits of the maximal dimension of the $G\times G$ action; these are orbits of quadruples with trivial isotropy groups. Therefore $\cF_4^\sharp(S\times S)$ is a manifold of dimension $2n$, where $n=\dim(\cF_4(S))$. If the action is free only on $\cF^\sharp_4(S\times S)$, the orbits of the remaining points are of dimension less than $2n$. This is exactly the case we study in Section \ref{sec:config}, i.e., the configuration space $\cF_4(\T)$ of quadruples of pairwise distinct points in the torus $\T=S^1\times S^1$. The subset of $\cF_4(\T)$ of maximal dimension two is isomorphic to 
$$
\cF_4^\sharp(\T)=\cF_4(S^1)\times\cF_4(S^1)=(\R\setminus\{0,1\})^2,
$$ 
a disconnected subset of $\R^2$ comprising nine connected components, see Theorem \ref{thm:vec}. Also, by considering the natural involution $\iota_0:\T\to \T$ which maps each $(x,y)$ to $(y,x)$ and the group $\overline{\cM(\T)}$ comprising M\"obius transformations of $\T$ followed by $\iota_0$, then by taking $\overline{\cF^\sharp_4(\T)}$ to be $\cC_4^\sharp(\T)$ cut by the diagonal action of $\overline{\cM(\T)}$, we find that it is isomorphic to a disconnected subset $\cQ$ of $\R^2$ comprising three open connected components and three components with 1-dimensional boundary (Theorem \ref{thm:band}).

At this point, the goal of parametrising the configuration space by cross-ratios defined on the torus has not yet been achieved; by Theorem \ref{thm:vec} the set $\cF_4(\T)$ admits a parametrisation obtained by assigning to each $\fp$ the pair $(\rX(\fx),\rX(\fy))$. To this end, for each $\fp\in\cC_4^\sharp$ we define
$$
\X(\fp)=\rX(\fx)\cdot \rX(\fy),
$$
see Section \ref{sec:realX}, which is $\cM(\T)$-invariant. Certain symmetries for $\X$ exist so that for each quadruple $\fp$ all 24 cross-ratios of quadruples resulting from permutations of points of $\fp$ are functions of two cross-ratios which we denote by $\X_1=\X_1(\fp)$ and $\X_2=\X_2(\fp)$. According to Proposition \ref{prop:fundX}, $(\X_1,\X_2)$ lie in a disconnected subset $\cP$ of $\R^2$ comprising six components. Three of these components are open and the remaining three have boundaries which are pieces of the parabola 
$$
\Delta(u,v)=u^2+v^2-2u-2v+1-2uv=0.
$$
In Theorem \ref{thm:F4} we prove that $\cF_4^\sharp(\T)$ is in a 2-1 surjection with $\cP$ and therefore $\overline{\cF_4^\sharp(\T)}$ is isomorphic to $\cP$. Remark that boundary components of $\cP$ correspond to quadruples $\fp$ such that all points of $\fp$ lie on a Circle, that is, a $\cM(\T)$-image of the diagonal curve $\gamma(x)=(x,x)$, $x\in S^1$, which is fixed by the involution $\iota_0$. Remark also that the parametrisation of $\overline{\cF_4^\sharp(\T)}$ by $\cQ$ and $\cP$ induces the same differentiable structure.

\medskip

We now discuss in brief some general aspects of M\"obius geometry. Let $S$ be a set comprising at least four points and denote by $\cC_4=\cC_4(S)$ the space of quadruples of pairwise distinct points of $S$. A {\it positive cross-ratio} $\bX$ on $\cC_4$ is a map $\cC_4\to\R_+$ such that for each $\fp=(p_1,p_2,p_3,p_4)\in\cS$, a list of symmetric properties hold; explicitly,
\begin{align*}
({\rm S1})&\quad \bX(p_1,p_2,p_3,p_4)=\bX(p_2,p_1,p_4,p_3)=\bX(p_3,p_4,p_1,p_2)=\bX(p_4,p_3,p_2,p_1),\\
({\rm S2})&\quad \bX(p_1,p_2,p_3,p_4)\cdot \bX(p_1,p_2,p_4,p_3)=1,\\
({\rm S3})&\quad \bX(p_1,p_2,p_3,p_4)\cdot \bX(p_1,p_4,p_2,p_3)\cdot \bX(p_1,p_3,p_4,p_2)=1.
\end{align*}
Hence all 24 real cross-ratios corresponding to a given quadruple $\fp$ are functions of $\bX_1(\fp)=\bX(p_1,p_2,p_3,p_4)$ and $\bX_2(\fp)=\bX(p_1,p_3,p_2,p_4)$. The {\it M\"obius structure} of $S$ is then defined to be the map
$$
\fM_S:\cC_4(S)\ni\fp\mapsto(\bX_1(\fp),\bX_2(\fp))\in(\R_+)^2.
$$
The {\it M\"obius group} $\fM(S)$ comprises bijections $g:S\to S$ that leave $\bX$ invariant, that is, $\bX(g(\fp))=\bX(\fp)$. We stress here that the above definitions vary depending on the author; however, all existing definitions are equivalent. An equivalent to our definition is the definition of {\it sub-M\"obius structure} in \cite{Bu}.

Frequently, a M\"obius structure is obtained from a metric (or even a semi-metric) $\rho$ on $S$ and it is called a {\it M\"obius structure associated to $\rho$}. In the primitive case of the circle $S^1$, from the real cross-ratio ${\rm X}$ we obtain a positive cross-ratio $\bX$ in $\cC_4(S^1)$ by assigning to each $\fp\in\cC_4(S^1)$ the number
$$
\bX(\fp)=|{\rm X}(\fp)|=\frac{|x_4-x_2||x_3-x_1|}{|x_4-x_1||x_3-x_2|}=\frac{\rho(x_4,x_2)\cdot \rho(x_3,x_1)}{\rho(x_4,x_1)\cdot \rho(x_3,x_2)}.
$$
The metric $\rho$ here is the extension of the euclidean metric in $\R$ to $\overline{\R}$: $\rho(x,y)=|x-y|$ if $x,y\in\R$, $\rho(x,\infty)=+\infty$ and $\rho(\infty,\infty)=0$. One verifies that the positive cross-ratio satisfies properties (S1), (S2), and (S3). The M\"obius structure of $S^1$ is thus the map
$$
\fM_{S^1}(\fp)=(\bX_1(\fp),\bX_2(\fp))\in(\R^+\setminus\{0,1\})^2.
$$
Note that the M\"obius group $\fM(S^1)$ for this M\"obius structure is ${\rm SL}(2,\R)$, a double cover of $\cM(S^1)$. Note also that since ${\rm X}_1(\fp)+{\rm X}_2(\fp)=1$, we have by the triangle inequality
\begin{equation}\label{eq:ptol}
\left|\bX_1(\fp)-\bX_2(\fp)\right|\le 1\quad\text{and}\quad\bX_1(\fp)+\bX_2(\fp)\ge 1.
\end{equation}
Explicitly, $\bX_1(\fp)-\bX_2(\fp)=1$ if $x_1$ and $x_3$ separate $x_2$ and $x_4$; $\bX_2(\fp)-\bX_1(\fp)=1$ if $x_1$ and $x_2$ separate $x_3$ and $x_4$; $\bX_1(\fp)+\bX_2(\fp)=1$ if $x_1$ and $x_4$ separate $x_2$ and $x_3$.

If the M\"obius structure $\fM_S$ of a space $S$ satisfies \eqref{eq:ptol}, then it is called Ptolemaean. The subsets of $S$ at which equalities hold in \eqref{eq:ptol} are called Ptolemaean circles. In this way the M\"obius structure of $S^1$ associated to the euclidean metric is Ptolemaean and $S^1$ itself is a Ptolemaean circle for this M\"obius structure. $S^1$ is the boundary of the hyperbolic disc $\bH_\C^1$; it is proved in \cite{P} that all M\"obius structures in boundaries of hyperbolic spaces $\bH_\K^n$, $n=1,2\dots$, are Ptolemaean; all these M\"obius structures are associated to the Korányi metric.

Therefore all boundaries of symmetric spaces of non-compact type and of rank-1 have M\"obius structures which are all associated to a metric and they are all Ptolemaean. In the case of the torus we study here, this does not happen: The M\"obius structure which is defined in Section \ref{sec:mob} 
\begin{enumerate}
\item is not associated to any semi-metric on $\T$;
\item is not Ptolemaean, but
\item there exist Ptolemaean circles for this structure.
\end{enumerate}
To the direction of defining M\"obius structures in boundaries of symmetric spaces of rank $>1$, little was known till recently; in his recent work \cite{B}, Beyrer explicitly constructs cross-ratio triples in Fürstenberg boundaries of symmetric spaces of higher rank. We have already mentioned that the torus $\T=S^1\times S^1$ which we study here appears naturally as the Fürstenberg boundary of the rank-2 symmetric space ${\rm SO}_0(2,2)/{\rm SO}(2)\times{\rm SO}(2)$ and is also isomorphic to the ideal boundary of 3-dimensional anti-de Sitter space $AdS^3$. Our results apply to these spaces, which we discuss in Section \ref{sec:cons}.

Finally, we note that the study of cross-ratios on $\T$ is strongly motivated by Higher Teichm\"uller Theory and the study of Anosov representations. Let $\Sigma$ be a closed surface group of genus $g \ge 2$. A representation $\rho = (\rho_1, \rho_2) : \pi_1(\Sigma) \to {\rm PSL}(2,\R) \times {\rm PSL}(2,\R)$ is Anosov if and only if both $\rho_1$ and $\rho_2$ are Fuchsian representations. Such representations admit continuous, $\rho$-equivariant boundary maps from the Gromov boundary $\partial_\infty \pi_1(\Sigma) \simeq S^1$ into the F\"urstenberg boundary $\T = \F(X)$. The product cross-ratio $\X$ introduced in this paper provides a natural invariant for quadruples of points in the limit set of such representations, offering a concrete geometric tool for studying the maximal components of the character variety ${\rm Hom}(\pi_1(\Sigma), {\rm SO}_0(2,2))/{\rm SO}_0(2,2)$.

\bigskip

\noindent{\it Acknowledgements:} The author was supported by the Medicus program, Grant Nr. 83765. Part of this work was carried out while the author visited the University of Zurich; hospitality is gratefully appreciated. The author also wishes to thank Viktor Schroeder and Jonas Beyrer for fruitful discussions.

\section{The Configuration Space of Four Points in the Torus}\label{sec:config}
Our main results lie in this section. In Section \ref{sec:trans} we study the transitive action of the group of M\"obius transformations of the torus. The results about the configuration space are in Sections \ref{sec:confT} and \ref{sec:realX}. 

\subsection{The action of M\"obius transformations on the torus}\label{sec:trans}
The torus $\T=S^1\times S^1$ is isomorphic to $\overline{\R}\times\overline{\R}$, where $\overline{\R}=\R\cup\{\infty\}$. Let $(x,y)\in\T$; a M\"obius transformation of $\T$ is a map $g:\T\to\T$ of the form
$$
g(x,y)=\left(g_1(x),g_2(y)\right),
$$
where $g_1$ and $g_2$ are in $\cM(S^1)$, that is,
$$
g_1(x)=\frac{ax+b}{cx+d},\quad g_2(y)=\frac{a'y+b'}{c'y+d'},
$$
where the matrices
$$
A_1=\begin{pmatrix}
a&b\\
c&d\end{pmatrix}\quad\text{and}\quad
A_2=\begin{pmatrix}
a'&b'\\
c'&d'\end{pmatrix}
$$
are both in ${\rm PSL}(2,\R)={\rm SL}(2,\R)/\{\pm I\}$. Thus the set of M\"obius transformations $\cM(\T)$ is $\cM(S^1)\times\cM(S^1)={\rm PSL}(2,\R)\times{\rm PSL}(2,\R)$. 

We wish to describe the action of $\cM(\T)$ on $\T$. First, the action is transitive; this follows directly from the 3-transitive action of $\cM(S^1)$ on $S^1$. Secondly, the action is not doubly-transitive in the usual sense. If 
$$
\fc=(p_1,p_2)=((x_1,y_1),(x_2,y_2))
$$
is a pair of distinct points on the torus, then the cases:
a) $x_1=x_2$ or $y_1=y_2$, and
b) $x_1\neq x_2$ and $y_1\neq y_2$,
are completely distinguished: A transformation $g\in\cM(\T)$ maps pairs of the form a) (resp. of the form b)) to pairs of the same form; this prevents $\cM(\T)$ from acting doubly-transitively on $\T$ in the usual sense. The doubly-transitive action of $\cM(\T)$ is rather partial in the sense above. Thirdly, as far as it concerns a triply-transitive action of $\cM(\T)$, distinguished cases appear again.
Indeed, consider an arbitrary triple 
$$
\ft=(p_1,p_2,p_3)=\left((x_1,y_1),(x_2,y_2),(x_3,y_3)\right)
$$
of pairwise distinct points in $\T$ and we have the following distinguished cases: 
\begin{enumerate}
\item [{a)}] Both $(x_1,x_2,x_3)$ and $(y_1,y_2,y_3)$ are triples of pairwise distinct points in $S^1$;
\item [{b)}] $y_1=y_2=y_3$ and $(x_1,x_2,x_3)$ is a triple of pairwise distinct points of $S^1$;
\item [{c)}] $x_1=x_2=x_3$ and $(y_1,y_2,y_3)$ is a triple of pairwise distinct points of $S^1$;
\item [{d)}] $x_i=x_j=x$, $x_l\neq x$, $i,j=1,2,3$, $i\neq j$, $l\neq i,j$, and $(y_1,y_2,y_3)$ is a triple of pairwise distinct points in $S^1$;
\item [{e)}] $y_i=y_j=y$, $y_l\neq y$, $i,j=1,2,3$, $i\neq j$, $l\neq i,j$, and $(x_1,x_2,x_3)$ is a triple of pairwise distinct points in $S^1$; 
\item [{f)}] Two $x_i$'s and two $y_j$'s are equal.
\end{enumerate}
All the above cases are not M\"obius equivalent: A $g\in\cM(\T)$ maps triples of each of the above categories to a triple of the same category. However, there is a triply-transitive action of $\cM(\T)$ at points which belong to the same category:
Notice for instance that in the first case we have that there exist $g_1$ and $g_2$ in $\cM(S^1)$ such that
$$
g_1(x_1)=g_2(y_1)=0,\quad g_1(x_2)=g_2(y_2)=\infty,\quad g_1(x_3)=g_2(y_3)=1.
$$
We derive that $g=(g_1,g_2)\in\cM(\T)$ satisfies
$$
g(\ft)=\left((0,0),(\infty,\infty),(1,1)\right).
$$
In case c), if $x_1=x_2=x_3$, then $(y_1,y_2,y_3)$ is a triple of distinct points in $S^1$. Therefore there exists a $g=(g_1,g_2)$ in $\cM(\T)$ such that $g_1(x_i)=0$, $g_2(y_1)=0$, $g_2(y_2)=\infty$, $g_2(y_3)=1$, that is,
$$
g(\ft)=\left((0,0),(0,\infty),(0,1)\right).
$$
Analogously for case b) where $y_1=y_2=y_3$, we find that there exists a $g\in\cM(\T)$ such that
$$
g(\ft)=\left((0,0),(\infty,0),(1,0)\right).
$$
The remaining cases are treated in the same manner.

\subsubsection{Circles}
For each $g=(g_1,g_2)\in\cM(\T)$ we get an embedding of $S^1$ into $\T$ which is given by the parametrisation
$$
\gamma(x)=(g_1(x),g_2(x)),\quad x\in S^1.
$$
Such embeddings of $S^1$ into $\T$ will be called {\it M\"obius embeddings of $S^1$} or {\it Circles} on $\T$. Notice first, that each Circle is the image of the {\it standard Circle} $R_0$ via an element of $\cM(\T)$; here, $R_0$ is the curve $\gamma(x)=(x,x)$, $x\in S^1$. Secondly, the involution $\iota_0$ of $\T$ defined by $\iota_0(x,y)=(y,x)$ fixes pointwise $R_0$. Hence to each Circle $R$ is associated an involution $\iota_R$ of $\T$ which fixes $R$ pointwise. Moreover, we have
\begin{prop}
Given a triple $\ft=(p_1,p_2,p_3)$ of the form a) above, there exists a Circle $R$ passing through the points of $\ft$ and thus the involution $\iota_R$ of $\T$ associated to $R$ fixes all points of $\ft$. 
\end{prop}
\begin{proof}
We normalise so that $\ft=(p_1,p_2,p_3)$ where
$$
p_1=(0,0),\quad p_2=(\infty,\infty),\quad p_3=(1,1).
$$
Then the Circle passing through $p_i$ is $R_0$ and the involution is $\iota_0$.
\end{proof} 
Three distinct points of $\T$ might lie on various embeddings of $S^1$; for instance, triples of the form b) and c) lie on $\gamma_y(x)=(g_1(x),y)$ for fixed $y$ and $\gamma_x(y)=(x,g_2(y))$ for fixed $x$, respectively, where $g_1,g_2\in\cM(S^1)$. But in any case, only triples of points of the form a) lie on Circles. 

\subsection{The configuration space of four points in $\T$}\label{sec:confT}
According to the notation set up in the introduction, let $\cC_4=\cC_4(\T)$ be the space of quadruples of pairwise distinct points in $\T$ and let $\cF_4=\cF_4(\T)$ be the {\it configuration space of quadruples of pairwise distinct points in} $\T$, that is, the quotient of $\cC_4$ by the diagonal action of the M\"obius group $\cM(\T)$ on $\cC_4$. Let $\fp=(p_1,p_2,p_3,p_4)\in\cC_4$ be arbitrary; if $p_i=(x_i,y_i)$, $i=1,2,3,4$, we shall denote by $\fx$ the quadruple $(x_1,x_2,x_3,x_4)$ and by $\fy$ the quadruple $(y_1,y_2,y_3,y_4)$. The isotropy group of $\fp$ is
\begin{align*}
\cM(\T)(\fp)&=\{g\in\cM(\T)\;|\;g(\fp)=\fp\}\\
&=\{(g_1,g_2)\in\cM(S^1)\times\cM(S^1)\;|\;g_1(\fx)=\fx\;\text{and}\;g_2(\fy)=\fy\}\\
&=\cM(S^1)(\fx)\times\cM(S^1)(\fy).
\end{align*}
Therefore the isotropy group $\cM(\T)(\fp)$ is trivial if and only if both isotropy groups $\cM(S^1)(\fx)$ and $\cM(S^1)(\fy)$ are trivial as well. If $\fp$ is such that $[\fp]$ is of maximal dimension (that is, both $[\fx]$ and $[\fy]$ are of maximal dimension), then we call $\fp$ {\it admissible}. Note that the dimension of the orbit of an admissible $\fp$ is 2. In the opposite case, we call $\fp$ {\it non-admissible}. We start with the non-admissible case first.

Let $\fp=(p_1,p_2,p_3,p_4)$, $\fx=(x_1,x_2,x_3,x_4)$ and $\fy=(y_1,y_2,y_3,y_4)$ as above. We distinguish the following cases for $\cM(S^1)(\fx)$:
\begin{enumerate}
\item [{$\fx$-1)}] $\cM(S^1)(\fx)$ is trivial and $\fx\in\cC_4(S^1)$. We may then normalise so that
$$
x_1=0,\quad x_2=\infty,\quad x_3=\rX(\fx),\quad x_4=1.
$$
\item [{$\fx$-2)}] $\cM(S^1)(\fx)$ is trivial and two points $x_i$ in $\fx$, $i\in\{1,2,3,4\}$, are equal. If for instance $x_1=x_2$, we may normalise so that 
$$
x_1=x_2=0,\quad x_3=1,\quad x_4=\infty;
$$ 
we normalise similarly for the remaining cases. 
\item [{$\fx$-3)}] $\cM(S^1)(\fx)$ is isomorphic to $\cM(S^1)(0,\infty)$: Three points $x_i$ in $\fx$, $i\in\{1,2,3,4\}$, are equal. If for instance $x_1=x_2=x_3$, we may normalise so that
$$
x_1=x_2=x_3=0,\quad x_4=\infty;
$$
we normalise similarly for the remaining cases.
\item [{$\fx$-4)}] $\cM(S^1)(\fx)$ is isomorphic to $\cM(S^1)(\infty)$: All points $x_i$ in $\fx$, $i=1,2,3,4$, are equal and we may normalise so that $x_i=\infty$.
\end{enumerate}
Notice that there are six sub-cases in $\fx$-2) and four sub-cases in $\fx$-3); in all, we have twelve distinguished cases. Entirely analogous distinguished cases $\fy$-1), $\fy$-2), $\fy$-3) and $\fy$-4) appear for $\cM(S^1)(\fy)$. Non-admissible quadruples $\fp$ are such that all combinations of cases for $\fx$ and $\fy$ may appear except when $\fx$ falls into the case $\fx$-1) and $\fy$ falls into the case $\fy$-1). Mind that not all combinations are valid; for instance, there can be no $\fp$ such that $\fx$ is as in $\fx$-4) and $\fy$ is as in $\fy$-2) or $\fy$-3). Subsets of $\cF_4$ corresponding to each valid combination are either of dimension 0 or 1. One-dimensional subsets appear when $\fx$ belongs to the $\fx$-1) case or $\fy$ belongs to the $\fy$-1) case. The corresponding subset is then isomorphic to $\R\setminus\{0,1\}$ together with a point. For clarity, we will treat two cases: First, suppose that the non-admissible $\fp$ is such that $\fx$ is as in $\fx$-1) and $\fy$ is as in $\fy$-2) with $y_1=y_2$. Then we may normalise so that
$$
p_1=(0,0),\quad p_2=(\infty, 0),\quad p_3=(\rX(\fx),1),\quad p_4=(1,\infty).
$$ 
Therefore the subset of $\cF_4(\T)$ comprising orbits of such $\fp$ is isomorphic to $\R\setminus\{0,1\}\times\{b_{12}\}$, where $b_{12}$ is the abstract point corresponding to quadruples $\fy$ such that $y_1=y_2$. Secondly, suppose that $\fp$ is such that $\fx$ is as in $\fx$-3) with $x_1=x_2=x_3$ and $\fy$ is as in $\fy$-2) with $y_3=y_4$. Then we may normalise so that
$$
p_1=(0,0),\quad p_2=(0, \infty),\quad p_3=(0,1),\quad p_4=(\infty,1).
$$ 
The corresponding subset of the orbit space is then isomorphic to $\{a_{123}\}\times \{b_{34}\}$, where $a_{123}$ is the abstract point corresponding to quadruples $\fx$ such that $x_1=x_2=x_3$ and $b_{34}$ is the abstract point corresponding to quadruples $\fy$ such that $y_3=y_4$. Table \ref{table:1} shows all distinguished subsets of $\cF_4(\T)$ comprising non-admissible orbits of quadruples $\fp$ such that $\fx$ and $\fy$ belong to the above categories: 

\begin{table}[h!]
\centering
\begin{tabular}{||c c c c||} 
 \hline
 $\fx$ & $\fy$ & Corresponding subset of $\cF_4$ & Components \\ [0.5ex] 
 \hline\hline
 $\fx$-1) & $\fy$-2) & $(\R\setminus\{0,1\})\times\{b_{ij}\}$ & 6 \\ 
 $\fx$-1) & $\fy$-3) & $(\R\setminus\{0,1\})\times\{b_{ijk}\}$ & 4 \\
 $\fx$-1) & $\fy$-4) & $(\R\setminus\{0,1\})\times\{\infty\}$ & 1 \\
 $\fx$-2) & $\fy$-1) & $\{a_{ij}\}\times(\R\setminus\{0,1\})$ & 6 \\ 
 $\fx$-2) & $\fy$-2) & $\{a_{ij}\}\times\{b_{kl}\}$ & 30 \\
 $\fx$-2) & $\fy$-3) & $\{a_{ij}\}\times\{b_{klm}\}$ & 12 \\
 $\fx$-3) & $\fy$-1) & $\{a_{ijk}\}\times(\R\setminus\{0,1\})$ & 4 \\
 $\fx$-3) & $\fy$-2) & $\{a_{ijk}\}\times\{b_{lm}\}$ & 12 \\
 $\fx$-4) & $\fy$-1) & $\{\infty\}\times(\R\setminus\{0,1\})$ & 1 \\ [1ex] 
 \hline
\end{tabular}
\caption{Subspaces of non-admissible orbits}
\label{table:1}
\end{table}

If now $\fp$ is an admissible quadruple, we have that both $\fx=(x_1,x_2,x_3,x_4)$ and $\fy=(y_1,y_2,y_3,y_4)$ are in $\cC_4(S^1)$. Let $\cC^\sharp_4=\cC_4^\sharp(\T)$ be the subspace of $\cC_4(\T)$ comprising admissible quadruples and denote by $\cF_4^\sharp=\cF_4^\sharp(\T)$ the corresponding orbit space. The bijection
$$
\fC:\cC_4^\sharp(\T)\ni \fp\mapsto (\fx,\fy)\in\cC_4(S^1)\times\cC_4(S^1),
$$
projects into the bijection
$$
\fF:\cF_4^\sharp(\T)\ni [\fp]\mapsto ([\fx],[\fy])\in\cF_4(S^1)\times\cF_4(S^1),
$$
and therefore we obtain:
\begin{thm}\label{thm:vec}
The configuration space $\cF_4(\T)$ of quadruples of pairwise distinct points of the torus $\T$ is isomorphic to a set comprising 77 distinguished components: 22 one-dimensional, 54 points, and a 2-dimensional subset corresponding to the subset $\cF_4^\sharp(\T)$ of admissible quadruples. This subset may be identified to $(\R\setminus\{0,1\})^2$. The identification is given by assigning to each $[\fp]$ the vector-valued cross-ratio $\vec\X(\fp)=(\rX(\fx),\rX(\fy))$.
\end{thm}

The set $\cF_4^\sharp=(\R\setminus\{0,1\})^2$ is a subset of $\R^2$ comprising nine connected open components. We consider the space $\overline{\cF_4^\sharp(\T)}$; this is $\cC_4^\sharp(\T)$ factored by the diagonal action of $\overline{\cM(\T)}$: The latter comprises elements of $\cM(\T)$ followed by the involution $\iota_0:(x,y)\mapsto(y,x)$ of $\T$.

We thus have:
\begin{thm}\label{thm:band}
The configuration space $\overline{\cF_4^\sharp(\T)}$ is identified to the disconnected subset $\cQ$ of $\R^2$ which is induced by identifying points of $(\R\setminus\{0,1\})^2$ which are symmetric with respect to the diagonal straight line $y=x$. Explicitly, $\cQ$ has three open components:
\begin{align*}
\cQ_1^0&=(-\infty,0)\times(0,1);\\
\cQ_2^0&=(-\infty,0)\times(1,+\infty);\\
\cQ_3^0&=(0,1)\times(1,+\infty),
\end{align*}
and three components with boundary:
\begin{align*}
\cQ_1^1&=\{(x,y)\in\R^2\;|\;x<0,\;x\le y<0\};\\
\cQ_2^1&=\{(x,y)\in\R^2\;|\;0<x<1,\;x\le y<1\};\\
\cQ_3^1&=\{(x,y)\in\R^2\;|\;x>1,\;y\ge x\}.
\end{align*}
\end{thm} 

\subsection{Real Cross-Ratios and another parametrisation}\label{sec:realX}
We define a real cross-ratio in $\cC_4^\sharp(\T)$ by
$$ 
\X(\fp)=\rX(\fx)\cdot \rX(\fy).
$$
One may show that all 24 cross-ratios corresponding to an admissible quadruple $\fp$ depend on the following two:
$$
\X_1(\fp)=[x_1,x_2,x_3,x_4]\cdot[y_1,y_2,y_3,y_4],\quad \X_2(\fp)=[x_1,x_3,x_2,x_4]\cdot[y_1,y_3,y_2,y_4].
$$
We now consider the map $\cG^\sharp:\cF_4^\sharp\to\R^2_*$, where
$$
\cG^\sharp([\fp])=\left(\X_1(\fp),\X_2(\fp)\right).
$$
The map $\cG^\sharp$ is well defined since $\X$ remains invariant under the action of $\cM(\T)$. Let
\begin{equation}\label{eq:P}
\cP=\{(u,v)\in(\R_*)^2\;|\;\Delta(u,v)=u^2+v^2-2u-2v+1-2uv\ge 0\}.
\end{equation}
The fundamental inequality for cross-ratios in the following proposition shows that $\cG^\sharp$ takes its values in $\cP$:

\begin{prop}\label{prop:fundX}
Let $\fp$ be an admissible quadruple of points in $\T$ and $\cP$ as in \eqref{eq:P}. Then
$$
(\X_1(\fp),\X_2(\fp))\in\cP.
$$
Moreover, $\Delta(\X_1(\fp),\X_2(\fp))=0$ if and only if all points of $\fp$ lie on a Circle.
\end{prop}
\begin{proof}
To prove the first statement, we may normalise so that $\fp=(p_1,p_2,p_3,p_4)$ where
$$
p_1=(0,0),\quad p_2=(\infty,\infty),\quad p_3=(x,y),\quad p_4=(1,1).
$$
We then calculate
$$
\X_1=xy,\quad \X_2=(1-x)(1-y).
$$
Therefore $\X_2=1+\X_1-x-y$, from where we derive
$$
1+\X_1-\X_2=x+y.
$$
Squaring both sides yields
$$
(1+\X_1-\X_2)^2=(x+y)^2\ge 4xy=4\X_1,
$$
and the inequality follows. For the second statement, observe that equality holds only if $x=y$, i.e. all points lie on the standard Circle $R_0$ on $\T$. 
\end{proof}

We proceed by showing that $\cG^\sharp$ is surjective: 
\begin{prop}
Let $(u,v)\in\cP$. Then there exists a $\fp\in\cC_4^\sharp(\T)$ such that
$$
\X_1(\fp)=u\quad\text{and}\quad \X_2(\fp)=v.
$$
\end{prop}
\begin{proof}
Since $\Delta=(1+u-v)^2-4u\ge 0$, there exist $x,y$ such that
$$
xy=u\quad\text{and}\quad x+y=1+u-v.
$$
In fact,
$$
x,y=\frac{1+u-v\pm\sqrt{\Delta}}{2}.
$$
Now one verifies that the admissible quadruple $\fp=(p_1,p_2,p_3,p_4)$ where
$$
p_1=(0,0),\quad p_2=(\infty,\infty),\quad p_3=(x,y),\quad p_4=(1,1),
$$
is the quadruple in question. The proof is complete.
\end{proof}

Let $\iota_0$ be the involution associated to the standard Circle $R_0$. Notice in the proof that the quadruple $\iota_0(\fp)$, that is,
$$
\iota_0(p_1)=(0,0),\quad \iota_0(p_2)=(\infty,\infty),\quad \iota_0(p_3)=(y,x),\quad \iota_0(p_4)=(1,1),
$$
also satisfies $\X_1(\iota_0(\fp))=u$ and $\X_2(\iota_0(\fp))=v$.

\begin{prop}
Suppose that $\fp$ and $\fp'$ are two quadruples in $\cC_4^\sharp(\T)$ such that
$$
\X_i(\fp)=\X_i(\fp'),\quad i=1,2.
$$
Then one of the following cases occurs:
\begin{enumerate}
\item There exists a $g\in\cM(\T)$ such that $g(\fp)=\fp'$;
\item There exists a $g\in\overline{\cM(\T)}$ such that $g(\fp)=\fp'$.
\end{enumerate}
\end{prop}
\begin{proof}
We may normalise so that
$$
p_1=(0,0),\quad p_2=(\infty,\infty),\quad p_3=(x,y),\quad p_4=(1,1),
$$
and
$$
p_1'=(0,0),\quad p_2'=(\infty,\infty),\quad p_3'=(x',y'),\quad p_4'=(1,1).
$$
Then $\X_i(\fp)=\X_i(\fp'),\; i=1,2,$ implies
$$
xy=x'y'\quad \text{and}\quad (1-x)(1-y)=(1-x')(1-y').
$$
It follows that $xy=x'y'$ and $x+y=x'+y'$. But then either $x=x'$ and $y=y'$, or $x=y'$ and $y=x'$. The proof is complete.
\end{proof}

The above discussion boils down to the following theorem:
\begin{thm}\label{thm:F4}
The $\cM(\T)$-(resp. $\overline{\cM(\T)}$)-configuration space $\cF_4^\sharp=\cF_4^\sharp(\T)$ (resp. $\overline{\cF_4^\sharp}=\overline{\cF_4^\sharp(\T)}$) of admissible quadruples of points in the torus $\T$ is in a 2-1 (resp. 1-1) surjection with the set $\cP$ given in \eqref{eq:P}. The subset $\cF_4^{\sharp,0}$ of both $\cF_4^\sharp$ and $\overline{\cF_4^\sharp}$ comprising equivalence classes of quadruples of points in the same Circle is in a bijection with the subset of $\cP$ comprising $(u,v)$ such that
$$
\Delta=u^2+v^2-2u-2v+1-2uv=0.
$$
Explicitly, $\cP$ has three open components
\begin{align*}
\cP_1^0&=(-\infty, 0)\times(0,+\infty);\\
\cP_2^0&=(-\infty, 0)\times(-\infty,0);\\
\cP_3^0&=(0,+\infty)\times(-\infty,0);
\end{align*}
and three components with one-dimensional boundaries:
\begin{align*}
\cP_1^1&=\{(u,v)\in\cP\;|\;0<u<1,\; 0<v<1,\;\Delta\ge 0\};\\
\cP_2^1&=\{(u,v)\in\cP\;|\;u>1,\;v>0,\;\Delta\ge 0\};\\
\cP_3^1&=\{(u,v)\in\cP\;|\;u>0,\;v>1,\;\Delta\ge 0\}.
\end{align*}
\end{thm}
\begin{rem}
The change of coordinates
$$
u=xy,\quad v=(1-x)(1-y)=1-x-y+xy
$$
maps the set $\cQ$ of Theorem \ref{thm:band} onto the set $\cP$ in a bijective manner.
\end{rem}

\begin{figure}[h!]
\centering
\setlength{\unitlength}{1cm}
\begin{picture}(8,8)(-2,-2)
    \put(-2,0){\vector(1,0){7}}
    \put(0,-2){\vector(0,1){7}}
    \put(5,-0.4){$u$}
    \put(-0.4,5){$v$}
    
    \qbezier(1,0)(0.2,0.2)(0,1)
    
    \put(2,2){$\cP_1^1 \; (u>0, v>0, \Delta \ge 0)$}
    \put(-1.5,2){$\cP_1^0$}
    \put(-1.5,-1.5){$\cP_2^0$}
    \put(2,-1.5){$\cP_3^0$}
    
    \put(1,0){\circle*{0.15}}
    \put(0,1){\circle*{0.15}}
    \put(1,-0.4){$1$}
    \put(-0.4,1){$1$}
\end{picture}
\caption{Schematic of the configuration space domain $\cP$. The boundary curve represents the parabola $\Delta(u,v)=0$, corresponding to quadruples lying on a Circle. The region inside the parabola ($0<u,v<1, \Delta<0$) is excluded.}
\label{fig:domainP}
\end{figure}

\section{M\"obius Structure}
Towards defining a M\"obius structure from the real cross-ratio $\X$ on the torus $\T$, we study first the case where both cross-ratios $\X_i(\fp)$, $i=1,2$, of an admissible quadruple of points are positive (Section \ref{sec:pos}). We then define $\fM_\T$ and prove in Section \ref{sec:mob} that it is not Ptolemaean.

\subsection{When both cross-ratios are positive}\label{sec:pos}
Let
$$
\cP^1=\cP_1^1\;\dot\cup\;\cP_2^1\;\dot\cup\;\cP_3^1.
$$
This set corresponds exactly to quadruples $\fp$ with both $\X_1(\fp)$ and $\X_2(\fp)$ positive. Equivalently, $\rX(\fx)$ and $\rX(\fy)$ belong to the same connected component of $\R\setminus\{0,1\}$, which means that the points of $\fx$ and $\fy$ have exactly the same type of ordering on the circle: If $x_1,x_2$ separate $x_3,x_4$, then also $y_1,y_2$ separate $y_3,y_4$ and so forth. Quadruples $\fp$ corresponding to $\cP^1$ have both $\X_1(\fp)$ and $\X_2(\fp)$ positive. Let $\cF^{\sharp,+}_4=(\cG^\sharp)^{-1}(\cP^1)$.

\begin{prop}\label{prop:Ptol-eq-T}
Let $\fp=(p_1,p_2,p_3,p_4)$ such that $[\fp]\in\cF^{\sharp,+}_4$ and let $\X_i=\X_i(\fp)$, $i=1,2$. Then
\begin{equation}\label{eq:X12}
\X_1^{1/2}+\X_2^{1/2}\ge 1\;\;\text{and}\;\;|\X_1^{1/2}-\X_2^{1/2}|\ge 1,\quad\text{or}\quad\X_1^{1/2}+\X_2^{1/2}\le 1\;\;\text{and}\;\;|\X_1^{1/2}-\X_2^{1/2}|\le 1.
\end{equation}
Moreover, $[\fp]\in\cF_4^{\sharp,0}$, that is, all points of $\fp$ lie in the same Circle if and only if
$$
\X_1^{1/2}+\X_2^{1/2}=1\quad \text{or}\quad |\X_1^{1/2}-\X_2^{1/2}|=1.
$$
Explicitly,
\begin{enumerate}
 \item $\X_1^{1/2}-\X_2^{1/2}=1$ if $p_1$ and $p_3$ separate $p_2$ and $p_4$;
 \item $\X_2^{1/2}-\X_1^{1/2}=1$ if $p_1$ and $p_2$ separate $p_3$ and $p_4$;
 \item $\X_1^{1/2}+\X_2^{1/2}=1$ if $p_1$ and $p_4$ separate $p_2$ and $p_3$.
\end{enumerate}
\end{prop}
\begin{proof}
From the fundamental inequality for cross-ratios (Proposition \ref{prop:fundX}) we have
\begin{align*}
0&\le\X_1^2+\X_2^2-2\X_1-2\X_2+1-2\X_1\X_2\\
&=(\X_1+\X_2-1)^2-4\X_1\X_2\\
&=(\X_1+\X_2-2\X_1^{1/2}\X_2^{1/2}-1)(\X_1+\X_2+2\X_1^{1/2}\X_2^{1/2}-1)\\
&=\left((\X_1^{1/2}+\X_2^{1/2})^2-1\right)\left((\X_1^{1/2}-\X_2^{1/2})^2-1\right),
\end{align*}
and this proves \eqref{eq:X12}. The details of the proof of the last statement are left to the reader.
\end{proof}

\subsection{M\"obius structure}\label{sec:mob}
From the real cross-ratio $\X:\cC^\sharp_4(\T)\to\R$ we define a positive cross-ratio $\bX:\cC^\sharp_4(\T)\to\R_+$ by setting
$$
\bX(\fp)=|\X(\fp)|^{1/2},
$$
for each $\fp\in\cC_4^\sharp$. The positive cross-ratio is invariant under $\cM(\T)$. The M\"obius structure on $\T$ associated to $\bX$ and restricted to $\cC_4^\sharp$ is the map
$$
\fM_\T:\cC^\sharp_4(\T)\ni\fp\mapsto(\bX_1(\fp),\bX_2(\fp)).
$$
Recall that $(S,\rho)$ is a pseudo-semi-metric space if $\rho:S\times S\to\R_+$ satisfies a) $x=y \implies \rho(x,y)=0$, and b) $\rho(x,y)=\rho(y,x)$, for all $x,y\in S$. The Möbius structure $\fM_\T$ is associated to the pseudo-semi-metric $\rho:\T\times\T\to\R_+$, given by
$$
\rho\left((x_1,y_1),(x_2,y_2)\right)=|x_1-x_2|^{1/2}\cdot|y_1-y_2|^{1/2},
$$
for each $(x_1,y_1)$ and $(x_2,y_2)$ in $\T$. 

It is instructive to contrast this with the rank-1 setting. In boundaries of rank-1 symmetric spaces, such as the Heisenberg group bounding $\bH_\C^2$, the natural Möbius structure is associated to the Korányi gauge, which respects a generalised triangle inequality and is strictly Ptolemaean. In the rank-2 setting of $\T$, the metric $\rho$ fails to be subadditive. Geometrically, this failure reflects the independent causal variations allowed along the product structure of $S^1 \times S^1$, which breaks the rigid Carnot-Carathéodory structure present in rank-1 boundaries. In Section \ref{sec:cons} we explain the representation-theoretic reason why we cannot have a natural positive cross-ratio compatible with the group action that is associated to any standard metric on $\T$. The following corollary concerning $\fM_\T$ follows directly from Proposition \ref{prop:Ptol-eq-T}.
\begin{cor}
The M\"obius structure $\fM_\T$ is not Ptolemaean. However, Circles are Ptolemaean circles for $\fM_\T$.
\end{cor}

\section{Application to Boundaries of ${\rm SO}_0(2,2)/{\rm SO}(2)\times{\rm SO}(2)$ and $AdS^3$}\label{sec:cons}
In this section we show how the torus appears as the Fürstenberg boundary of the symmetric space ${\rm SO}_0(2,2)/{\rm SO}(2)\times{\rm SO}(2)$ as well as the ideal boundary of the 3-dimensional anti-de Sitter space $AdS^3$. We refer to \cite{BJ} for compactifications of symmetric spaces, to \cite{B} for recent developments on M\"obius structures in Fürstenberg boundaries, and finally to \cite{D} for a comprehensive treatment of anti-de Sitter space and its relations to Hyperbolic Geometry.  

Let $\R^{2,2}=\R^4\setminus\{0\}$ be the real vector space of dimension 4 equipped with a non-degenerate, indefinite pseudo-hermitian form $\langle\cdot,\cdot\rangle$ of signature $(2,2)$. Such a form is given by a $4\times 4$ matrix with 2 positive and 2 negative eigenvalues.
Let $\bx=\begin{bmatrix} x_1 & x_2 &x_3 &x_4\end{bmatrix}^T$ and $\by=\begin{bmatrix} y_1 & y_2 &y_3 &y_4\end{bmatrix}^T$ be column vectors. The pseudo-hermitian form is then defined by
$$
\langle\bx,\by\rangle=x_1y_4-x_2y_3-x_3y_2+x_4y_1
$$
and it is given by the matrix
$$
J=\begin{pmatrix}
0&0&0&1\\
0&0&-1&0\\
0&-1&0&0\\
1&0&0&0
\end{pmatrix}.
$$
The isometry group of this pseudo-hermitian form is $G={\rm SO}_0(2,2)$. There is a natural identification of $G$ with ${\rm SL}(2,\R)\times{\rm SL}(2,\R)$, see \cite{GPP}; if
$$
A_1=\begin{pmatrix} a_1&b_1\\ c_1&d_1\end{pmatrix}\quad\text{and}\quad
A_2=\begin{pmatrix} a_2&b_2\\ c_2&d_2\end{pmatrix}\in{\rm SL}(2,\R),
$$
then the pair $(A_1,A_2)$ is identified to
$$
A=\begin{pmatrix}
a_1A_2^{-1}&b_1A_2^{-1}\\
c_1A_2^{-1}&d_1A_2^{-1}\end{pmatrix}\in{\rm SO}_0(2,2).
$$
The group $K={\rm SO}(2)\times{\rm SO}(2)$ is a maximal compact subgroup of $G$ and $X=G/K$ is a symmetric space of rank-2. The symmetric space $X$ is also realised as $\bH_\C^1\times\bH_\C^1$, where $\bH_\C^1=D$ is the Poincaré hyperbolic unit disc. The torus $\T=S^1\times S^1$ is the {\it maximal Fürstenberg boundary} $\mathbb{F}(X)$ of the symmetric space $X$. Recall that if $G$ is a connected semi-simple Lie group and $X=G/K$ is the associated symmetric space, then the maximal Fürstenberg boundary $\mathbb{F}(X)$ may be thought of as $G/P_0$, where $P_0$ is a minimal parabolic subgroup of $G$, see for instance \cite{BJ}. If $X$ is of rank $\ge 2$, then $\mathbb{F}(X)$ cannot be the whole boundary of any compactification of $X$. In particular, if $X=D\times D$, then $\mathbb{F}(X)={\rm SO}_0(2,2)/P_0$, $P_0=AN\times AN$ where $AN$ is the $AN$ group in the Iwasawa $KAN$ decomposition of ${\rm SL}(2,\R)$. In this manner $\mathbb{F}(X)$ is just the corner of the boundary $(\overline{D}\times S^1)\cup(S^1\times\overline{D})$ of the compactification $\overline{D}\times\overline{D}$ of $X$. 

A rather neat way to represent $\T=\F(X)$ is via its isomorphism to the ideal boundary of anti-de Sitter space, which is obtained as follows:
For the pseudo-hermitian product there are subspaces of positive (space-like) vectors $V_+$, of null (light-like) vectors $V_0$, and of negative (time-like) vectors $V_-$:
\begin{align*}
V_+&=\left\{\bx\in\R^{2,2}\;|\;\langle\bx,\bx\rangle> 0\right\},\\
V_0&=\left\{\bx\in\R^{2,2}\;|\;\langle\bx,\bx\rangle=0\right\},\\
V_-&=\left\{\bx\in\R^{2,2}\;|\;\langle\bx,\bx\rangle<0\right\}.
\end{align*}
If $\lambda$ is a non-zero real, then $\langle\lambda\bx,\lambda\bx\rangle=\lambda^2\langle\bx,\bx\rangle$. Therefore $\lambda\bx$ is positive, null or negative if and only if $\bx$ is positive, null or negative, respectively. Let $P$ be the projection map from $\R^{2,2}$ to projective space $P\R^3$. The {\it projective model of anti-de Sitter space} $AdS^3$ is now defined as the collection of negative vectors $PV_-$ in $P\R^3$ and its {\it ideal boundary} $\partial_\infty AdS^3$ is defined as the collection $PV_0$ of null vectors. Anti-de Sitter space $AdS^3$ carries a natural Lorentz structure; the isometry group of this structure is the projectivisation of the set ${\rm SO}(2,2)$ of unitary matrices for the pseudo-hermitian form with matrix $J$, that is, ${\rm PSO}_0(2,2)={\rm SO}(2,2)/\{\pm I\}$; here $I$ is the identity $4\times 4$ matrix. From the discussion above we have that ${\rm PSO}_0(2,2)$ is identified to ${\rm PSL}(2,\R)^2=\cM(\T)$. Now the identification of $\partial_\infty AdS^3$ with the torus $\T=\F(X)$ is given in terms of the Segre embedding $\cS:\R P^1\times\R P^1\to\R P^3$. Recall that in homogeneous coordinates the Segre embedding $w=\cS(x,y)$ is defined by 
$$
(\bx,\by)=\left(\begin{bmatrix} x_1\\ x_2 \end{bmatrix}\;,\; \begin{bmatrix} y_1\\ y_2 \end{bmatrix}\right)\mapsto \bw=\begin{bmatrix} x_1y_1\\ x_1y_2\\ x_2y_1\\ x_2y_2 \end{bmatrix}.
$$ 
Notice that $\bw$ is a null vector. The action of the isometry group ${\rm PSO}_0(2,2)={\rm PSL}(2,\R)^2$ of $AdS^3$ is extended naturally on the ideal boundary $\partial_\infty AdS^3$, which in this manner is identified to $\T$. 

We stress at this point that in contrast with the case of hyperbolic spaces, there are distinct points in $\partial_\infty AdS^3$ which may be orthogonal. To see this, let $p=\cS(x,y)$ and $p'=\cS(x',y')$ be any distinct points; if 
$$
\bx=\begin{bmatrix} x_1\\x_2\end{bmatrix},\quad
\by=\begin{bmatrix} y_1\\y_2\end{bmatrix},\quad
\bx'=\begin{bmatrix} x_1'\\x_2'\end{bmatrix},\quad
\by'=\begin{bmatrix} y'_1\\y'_2\end{bmatrix},
$$
then  
$$
\bp=\begin{bmatrix}x_1y_1\\ x_1y_2\\ x_2y_1\\ x_2y_2 \end{bmatrix},\quad\bp'=\begin{bmatrix} x_1'y_1'\\ x_1'y_2'\\ x_2'y_1'\\ x_2'y_2'\end{bmatrix}.
$$
One then calculates
$$
\langle\bp,\bp'\rangle=(x_1x_2'-x_2x_1')(y_1y_2'-y_2y_1').
$$
Therefore $\langle\bp,\bp'\rangle=0$ if and only if $x=x'$ or $y=y'$. For fixed $p\in\partial_\infty AdS^3$, $p=\cS(x,y)$, the locus 
$$
p^c=\{\cS(x,z)\;|\;z\in S^1\}\cup \{\cS(z,y)\;|\;z\in S^1\}\equiv (\{x\}\times S^1)\cup(S^1\times\{y\})
$$ 
comprises points of the ideal boundary which are orthogonal to $p$. We call $p^c$ the {\it cross-completion} of $p$; transferring the picture to the context of a point $p=(x,y)$ lying on $\T$, orthogonal points are all points of $(\{x\}\times S^1)\cup(S^1\times\{y\})$. This orthogonality condition yields a direct Lorentzian interpretation of the 77-component orbit stratification of $\cF_4(\T)$ derived in Theorem \ref{thm:vec}. In the language of $AdS^3$ geometry, non-admissible quadruples correspond to configurations where at least two points are lightlike separated. For instance, if $\fp \in \cC_4(\T)$ is a quadruple where $x_i = x_j$ or $y_i = y_j$, the corresponding ideal points $\cS(x_i, y_i)$ and $\cS(x_j, y_j)$ lie on the same generator of the lightcone at infinity. The lower-dimensional orbit components detailed in Table \ref{table:1} thus perfectly classify the moduli of causal, lightlike degenerations of four points on the ideal boundary $\partial_\infty AdS^3$. The ideal boundary $\partial_\infty AdS^3$ may thus be thought of as the union of the cross-completion of $\infty=\cS(\infty,\infty)$ and the remaining region of the torus which we denote by $N$. The set $N$ comprises points $p=\cS(x,y)$, $x,y\neq \infty$, with standard lifts
$$
\bp=\begin{bmatrix}xy&x&y&1\end{bmatrix}^T,
$$ 
and can be viewed as the saddle surface $x_1=x_2x_3$ embedded in $\R^3$. But actually, $N$ admits a group structure: First, if $p=\cS(x,y)\in N$, we call $(x,y)$ the $N$-coordinates of $p$. To each such $p$ we assign the matrix
$$
T(x,y)=\begin{pmatrix}
1&y&x&xy\\
0&1&0&x\\
0&0&1&y\\
0&0&0&1
\end{pmatrix},
$$
whose projectivisation gives an element of ${\rm PSO}_0(2,2)$ in the unipotent isotropy group of $\infty$. Note that if $G$ is the isomorphism ${\rm SL}(2,\R)^2\to{\rm SO}_0(2,2)$, and $KAN$ is the Iwasawa decomposition of ${\rm SL}(2,\R)$, then $T(x,y)$ lies in the image $G(N,N)$. It is straightforward to verify that $T(x,y)$ leaves the cross-completion $\infty^c$ of infinity invariant and maps $o=\cS(0,0)$ to $p$. 

Also, for $p=(x,y)$ and $p'=(x',y')$, we have
$$
T(x,y)T(x',y')=T(x+x',y+y'),\quad \left(T(x,y)\right)^{-1}=T(-x,-y).
$$
Thus $T$ is a group homomorphism from $\R^2$ to ${\rm PSO}_0(2,2)$ with group law
$$
(x,y)\star(x',y')=(x+x',y+y').
$$
In other words, $N$ admits the structure of the additive group $(\R^2,+)$. The natural Euclidean metric $e:N\times N\to\R_+$ where
$$
e\left((x,y),(x',y')\right)=((x-x')^{2}+(y-y')^2)^{1/2},
$$ 
is invariant under the left action of $N$, but its similarity group is not ${\rm PSO}_0(2,2)$. To see this, consider
$$
D_\delta=\begin{pmatrix} \delta&0\\ 0&1/\delta\end{pmatrix}\quad \text{and}\quad
D_{\delta'}=\begin{pmatrix} \delta'&0\\ 0&1/\delta'\end{pmatrix},\quad\delta,\delta'>0.
$$ 
Then $G(D_\delta,D_{\delta'})(\xi_1,\xi_2)=(\delta^2\xi_1,(1/\delta')^2\xi_2)$, and $A=G(D_\delta,D_{\delta'})$ does not scale $e$ unless $\delta\delta'=1$, which is not always the case. Since all metrics in $\R^2$ are equivalent to $e$, there is no natural metric in $N$ whose similarity group equals ${\rm PSO}_0(2,2)$. In contrast, we define a function $a:N\to\R$,
$$
a(x,y)=xy
$$
and a gauge
$$
\|(x,y)\|=|a(x,y)|^{1/2}=|x|^{1/2}|y|^{1/2}.
$$
Essentially, we are mimicking here Korányi and Reimann and their construction for the Heisenberg group case, see \cite{KR}. The pseudo-semi-metric $\rho:N\times N\to\R_+$ is then defined by
$$
\rho\left((x,y),(x',y')\right)=\|(x',y')^{-1}\star(x,y)\|=|x-x'|^{1/2}|y-y'|^{1/2}.
$$ 
Let $\overline{N}=N\cup\{\infty\}$. The set $\cC_4^\sharp(\T)$ of admissible quadruples is actually the set $\cC_4^\sharp(\overline{N})$ of quadruples of points of $\overline{N}$ such that none of these points belongs to the cross-completion of any other point in the quadruple. The configuration space $\cF_4^\sharp(\T)$ is thus identified to $\cC_4^\sharp(\overline{N})$ cut by the diagonal action of ${\rm PSO}_0(2,2)$ and the configuration space $\overline{\cF_4^\sharp(\T)}$ is $\cC_4^\sharp(\overline{N})$ cut by the diagonal action of $\overline{{\rm PSO}_0(2,2)}$, which comprises elements of ${\rm PSO}_0(2,2)$ followed by the involution $\iota_0:(x,y)\mapsto(y,x)$ of $\overline{N}$. 

The real cross-ratio $\X$ is thus defined in $\cC_4^\sharp(\overline{N})$ by
$$
\X(\fp)=\frac{a(p_4\star p_2^{-1})\cdot a(p_3\star p_1^{-1})}{a(p_4\star p_1^{-1})\cdot a(p_3\star p_2^{-1})},
$$
for each $\fp\in\cC_4^\sharp(\overline{N})$. The positive cross-ratio is
$$
\bX(\fp)=\frac{\rho(p_4,p_2)\cdot\rho(p_3,p_1)}{\rho(p_4,p_1)\cdot\rho(p_3,p_2)}.
$$
The results of the previous section apply now immediately.

\end{document}